\documentclass[a4paper,11pt]{article}

\usepackage{times, named}
\usepackage{stmaryrd, amssymb, amsthm}

\title{Product closure of some \\ second-order modal logics}
\author{Jonathan Zvesper\footnote{Oxford University Computing Laboratory, Parks Road, Oxford OX1 3QD, UK \texttt{jonathan.zvesper@comlab.ox.ac.uk}}}

\newtheorem{theorem}{Theorem}
\newtheorem{defined}[theorem]{Definition}

\newtheorem{exa}[theorem]{Example}

\newtheorem{proposition}[theorem]{Proposition}

\newtheorem{lemma}[theorem]{Lemma}
\newtheorem{corollary}[theorem]{Corollary}
\newtheorem{remark}[theorem]{Remark}
\newtheorem{fact}[theorem]{Fact}

\newcommand{\implies}{\ensuremath{\supset}}

\newcommand{\dfn}[1]{\stackrel{\textup{\tiny df}}{#1}}
\newcommand{\df}[1]{\textbf{\emph{#1}}}
\newcommand{\restr}{|}

\begin{document}

\maketitle

\begin{abstract}

Product update is an operation on models introduced into epistemic logic in order to represent a broad class of informational events.
If adding modalities representing product update to a language does not alter its expressive power
then we say that the language is `\emph{closed for product update}'. 
The basic modal language is known to be closed for product update \cite{BMS, Gerbrandy_PhD}.
We establish that monadic second order logic is closed for product update (Theorem \ref{thm:prod}).
Our technique is to pass via an intermediate language with what we call `action nominals'.
We obtain as corollaries that propositionally quantified modal logic is closed for product update, as is the modal $\mu$-calculus.

\end{abstract}


\section{Preamble}

The semantic operation of \emph{product update} was introduced into epistemic logic as a 
generalisation of relativisation.
Where logics including relativisation operators can reason about \emph{public announcements},
logics with `action' product update operators enable us to reason about information change under a much broader class of informational events,
allowing for arbitrary levels of uncertainty amongst the agents modelled.
Adding action operators for product update into formal languages
therefore opened the door for the logical analysis of social and formal protocols,
including situations of deception and suspicion,
and more generally any kind of uncertainty concerning what informational event is taking place.

It has been established \cite{BMS} that
adding such action operators does not increase the expressivity of the basic modal language,
but does increase the expressivity of the language with common knowledge.
A natural question when looking at a modal language is whether adding operators for product update increases the expressivity or not.
We answer here three instances of that question,
for very expressive modal languages,
each of which has some sort of second-order quantification,
and so is not less expressive than the first-order correspondence language.
These languages are those of:
monadic second order logic (MSO), propositionally quantified modal logic, and the modal $\mu$-calculus.
In each case we find that the logics are closed for product update,
i.e.\ that adding action modalities does not increase the expressive power over the class of all (relational) models.
Our method is first to prove that MSO is product closed, for which we introduce a novel technique involving `action nominals';
and then to use characterisation results of the other two languages as fragments of MSO
to obtain the results concerning them as corollaries of the main result concerning MSO.

\section{Product closure}

We define in this section what we mean by `closure for relativisation' and `closure for product update'.
These two `closures' of a logic%
\footnote{Or, as shall say, of a \emph{language}, taking it for granted that we are talking about the logic of that language over the class of all relational models.}
are best defined as in terms of two semantic operations on its models.

If $\mathcal{M} = (\Omega, R, V)$ is a relational model\footnote{So with $\Omega$ a non-empty set, $R \subseteq \Omega \times \Omega$ and $V: W \rightarrow 2^W$.} and $A \subseteq \Omega$, we write $\mathcal{M}\restr A$ for the \df{relativisation} of $\mathcal{M}$ to $A$, defined as $(A, R\restr A, V\restr A)$, where $R\restr A \dfn{=} R \cap (A \times A)$ and $(V\restr A)(p) \dfn{=} V(p) \cap A$.
Then we say that a language $\mathcal{L}$ is \df{closed for relativisation} just when for any pair of formulas $\{\varphi, A\} \in \mathcal{L}$, there is a formula $\psi \in \mathcal{L}$ such that for any model $\mathcal{M}$, $\llbracket \psi \rrbracket^\mathcal{M} = \llbracket \varphi \rrbracket^{\mathcal{M} \restr \llbracket A \rrbracket^\mathcal{M}}$.  To spell this definition out: for any formula $\varphi$, there is a formula $\psi$ that is true in  model $\mathcal{M}$ just when $\varphi$ `will be' true in the relativised model $\mathcal{M} \restr \llbracket A \rrbracket^\mathcal{M}$.  A natural interpretation of relativisation in modal models where the accessibility relation is taken to be \emph{epistemic} is of an action of public announcement.  Public announcements were introduced and discussed in this context in \cite{Plaza_PA, GG, vB_Lonely}.  To repeat the temporal idiom, $\psi$ can be thought of as saying that $\varphi$ \emph{will be} the case \emph{after} announcement of $A$.

To say that a language $\mathcal{L}$ is closed for relativisation is equivalent to saying that if we were to enrich $\mathcal{L}$ with `relativisation modalities' then we would obtain a language with exactly the same expressive power as $\mathcal{L}$.  That is, consider another language $\mathcal{L}_!$, obtained from $\mathcal{L}$ by addition of a family of modalities $\langle ! A \rangle$ for each $A \in \mathcal{L}$, endowed with the following semantics:
\[
\mathcal{M}, \omega \vDash \langle !\psi \rangle \varphi \dfn{\Leftrightarrow} \mathcal{M}, \omega \vDash \psi \textup{ and } \mathcal{M} \restr \llbracket \psi \rrbracket^\mathcal{M}, \omega \vDash \varphi
\]
$\mathcal{L}$ is closed for relativisation iff $\mathcal{L}_!$ has precisely the same expressive power as $\mathcal{L}$.

One can show by a `compositional analysis' of its semantics that,
for example, the basic modal language $\mathcal{L}_\square$ is closed for relativisation.
That is, one gives so-called `reduction axioms', validities of the form $\langle !\psi \rangle O(\varphi) \equiv Q(\langle !\psi \rangle \varphi)$, for each connective $O$ of the language.
This technique, introduced for $\mathcal{L}_\square$ in \cite{Plaza_PA},
shows not only that a language is closed for relativisation but, when $Q$ is computable from $O$, that it is \emph{computably} so closed.

A similar technique can be employed to show that a language is closed for `product update'.  Product update was developed in the context of epistemic logic in \cite{Gerbrandy_PhD, BMS}.  Semantically, it is a more complicated operation than relativisation.  In order to define product update we need to define `event models', defined in terms of the language $\mathcal{L}$.  An \df{event model} is a tuple $(\Sigma, E, PRE)$, where $\Sigma$ is a finite non-empty set of `events', $E \subseteq \Sigma \times \Sigma$ is a relation over $\Sigma$ and $PRE_\_: \Sigma \rightarrow \mathcal{L}$ gives the `precondition' formula for each `event'.  Product update $\otimes$ is then a function that, given a model and an event model, returns a new model.
It is defined as follows:
\[
\mathcal{M} \otimes \mathcal{A} \dfn{=}
\left(\begin{array}{l}
\{(\omega, \alpha) \in \Omega \times A \mid \mathcal{M}, \omega \vDash Pre_\alpha \}, \\
\{ ((\omega, \alpha), (\omega', \alpha')) \mid \omega R \omega'~\&~\alpha E  \alpha' \}, \\
\{ (p, \{(\omega, \alpha) \mid \omega \in V(p) \}) \mid p \in PROP \} \\
\end{array}
\right)
\]
For any event model $\mathcal{A} = (\Sigma, E, PRE)$, we could expand the language $\mathcal{L}$ to a language $\mathcal{L}_\alpha$ by adding a collection of operators $\langle \alpha \rangle$ for each event $\alpha \in \Sigma$, with the following semantics:
\[
\mathcal{M}, \omega \vDash \langle \alpha \rangle \varphi \dfn{\Leftrightarrow} \mathcal{M}, \omega \vDash Pre_\alpha \textup{ and } \mathcal{M} \otimes \mathcal{A}, (\omega, \alpha) \vDash \varphi,
\]
We say that a language $\mathcal{L}$ is \df{closed for product update} just if any language $\mathcal{L}_\alpha$ formed in this way has the same expressive power as $\mathcal{L}$.  Equivalently (for all languages $\mathcal{L}$ that we consider), $\mathcal{L}$ is closed for product update just if for every $\psi \in \mathcal{L}$ and each event $\alpha$ there is some $\varphi \in \mathcal{L}$ such that in any model, $\llbracket \langle \alpha \rangle \psi \rrbracket = \llbracket \varphi \rrbracket$.

\section{MSO}

We now define the first logic that we will consider:
propositionally quantified modal logic with a global modality.
The language $\mathcal{L}_{\exists, U}$ is given by the following Backus-Naur form:
\[
 \varphi ::= p \mid \neg \varphi \mid \varphi \land \varphi \mid \square \varphi \mid \exists p. \varphi \mid U \varphi
\]
We assume, throughout the paper, standard abbreviations like $\forall$, $\implies$, $\lor$.
If $\mathcal{M} = (\Omega, R, V)$, then `$\mathcal{M}[p \mapsto X]$' denotes the model $(\Omega, R, V')$,
where $V'(p) = X$ and for all $q \neq p$, $V'(q) = V(q)$.
The semantics of $\mathcal{L}_{\exists, U}$ are as follows, where $\mathcal{M} = (\Omega, R, V)$ is a model.
\begin{eqnarray*}
\mathcal{M}, \omega \vDash p				& \dfn \Leftrightarrow & \omega \in V(p) \\
\mathcal{M}, \omega \vDash \neg \varphi		& \dfn \Leftrightarrow & \mathcal{M}, \omega \nvDash \varphi \\
\mathcal{M}, \omega \vDash \varphi \land \psi	& \dfn \Leftrightarrow & \mathcal{M}, \omega \vDash \varphi \textup{ and } \mathcal{M}, \omega \vDash \psi \\
\mathcal{M}, \omega \vDash \square \varphi	& \dfn \Leftrightarrow & \forall \omega' (\omega R \omega' \Rightarrow \omega' \vDash \varphi)\\
\mathcal{M}, \omega \vDash \exists p. \varphi	& \dfn{\Leftrightarrow} & \exists X \subseteq \Omega: \mathcal{M}[p \mapsto X], \, \omega \vDash \varphi \\
\mathcal{M}, \omega \vDash U \varphi		&\dfn{\Leftrightarrow} & \forall \omega' \in \Omega,\, \omega' \vDash \varphi
\end{eqnarray*}
The same language but without the universal modality $U$ was first studied in \cite{Fine_QML}, and we will consider it below in section \ref{sec:qml}.
The language $\mathcal{L}_{\exists, U}$ is known to be a translation into modal notation of monadic second-order logic $\mathcal{L}_{MSO}$.
That is: for every $\mathcal{L}_{MSO}$-formula $\varphi(x)$ with one free variable (``$\mathcal{L}_{MSO}$-type'')
there is a $\mathcal{L}_{\exists, U}$-formula true at those points of which $\varphi(x)$ is true, and vice-versa.
This is the recursive definition of $\mathcal{L}_{MSO}$:
\[
\varphi ::= Px \mid xRy \mid \neg \varphi \mid \varphi \land \varphi \mid \exists x \varphi \mid \exists P \varphi
\]
The semantics of the first-order part of $\mathcal{L}_{MSO}$ are standard,
and the propositional quantifier $\exists P$ has the natural interpretation,
as quantifying unrestrictedly over subsets of the domain.
The only slightly tricky clause of the translation between $\mathcal{L}_MSO$ and $\mathcal{L}_{\exists, U}$ is the following:
\[TR(\exists x \varphi(x)) = \exists p. (E p \land \forall q. (U(q \implies p) \implies U(p \implies q)) \land \varphi(p)).\]
This equivalence means that we prove about the expressivity of $\mathcal{\exists, U}$
can equivalently be read as being about $\mathcal{L}_{MSO}$-types, and we can refer to $\mathcal{L}_{\exists, U}$ as ``MSO''.

\section{MSO is product closed}

To warm up we will remark that $\mathcal{L}_{\exists, U}$ is closed under relativisation.
In order to do this we want a `reduction axiom' for the quantifier $\exists p$.
The following Fact states such a reduction axiom:
\begin{fact}
\label{fct:MSORel}
If $p$ does not occur in $A$, then:
\[
\vDash \langle !A \rangle \exists p. \varphi \equiv \exists p. (U(p \implies A) \land \langle !A \rangle \varphi)
\]
\end{fact}
Fact \ref{fct:MSORel} plays the central part in the proof of Proposition \ref{prop:rel}.
\begin{proposition}
\label{prop:rel}
$\mathcal{L}_{\exists, U}$ is computably closed for relativisation.
\end{proposition}
(Proposition \ref{prop:rel} follows from Fact \ref{fct:MSORel}, and in any case would be an immediate corollary of Theorem \ref{thm:prod} below.)

In the rest of this section we will work towards Theorem \ref{thm:prod}, which states that $\mathcal{L}_{\exists, U}$ is also (computably) closed for product update.
Fix some event model $(\Sigma, E, PRE)$ with $\Sigma = \{ \alpha_0,\ldots,\alpha_{n-1}\}$.  We will show that the language $\mathcal{L}_{\exists, U, \alpha}$ obtained by adding event modalities $\langle \alpha \rangle$ for each $\alpha \in \Sigma$ has precisely the same expressive power as the original language $\mathcal{L}_{\exists, U}$.

In order to show that $\mathcal{L}_{\exists, U}$ is closed for relativisation we were able to give a reduction axiom for the quantifier $\exists p$.
The case for product update is a little bit more subtle.
We will give a reduction axiom for the quantifier $\exists p$,
but we do not see how to do this directly in the language $\mathcal{L}_{\exists, U}$.
Therefore we take a detour via some additional rather artificial vocabulary that allows expressing a subset of the product space in terms of a sequence of subsets of the initial model.

This new vocabulary consists of nullary modalities $j_0, \ldots, j_{n}$.
Intuitively speaking, these action nominals will say, in the model $\mathcal{M} \times \mathcal{A}$, that action $\alpha_i$ has \textbf{just occurred}.  We call these $j_i$'s `\df{action nominals}'.
So the language $\mathcal{L}_{\alpha, j}$ is formed from the language $\mathcal{L}_{\alpha}$ by including the action nominals, but \emph{only when they occur under the scope of an action modality} $\langle \alpha \rangle$.
To formally specify the language $\mathcal{L}_{\exists, U, \alpha, j}$,
we consider first the set of sentences $\mathcal{S}_{\exists, U, \alpha, j}$,
for which we cannot formulate a semantics, and therefore do not use the word `language'.
\[
 \varphi ::= p \mid j_i \mid \neg \varphi \mid \varphi \land \varphi \mid \square \varphi \mid \exists p. \varphi \mid U \varphi \mid \langle \alpha \rangle \varphi.
\]
The problem in giving a semantics for $\mathcal{S}_{\exists, U, \alpha, j}$
is precisely that action nominals can occur outside of the scope of action modalities,
in which case we have no way to evaluate them.%
\footnote{One natural option might be to say that in general $\mathcal{M}, \omega \vDash j_i$ holds when there is some model $\mathcal{M}'$ with a point $\omega'$ such that $(\mathcal{M}' \otimes \mathcal{A}, (\omega', \alpha_i))$ is isomorphic to $(\mathcal{M}, \omega)$, but we have no need for such an option here.}
So we define the language $\mathcal{L}_{\exists, U, \alpha, j}$ as follows, where $\psi$ can take values in $\mathcal{S}_{\exists, U, \alpha, j}$:
\[
 \varphi ::= p \mid \neg \varphi \mid \varphi \land \varphi \mid \square \varphi \mid \exists p. \varphi \mid U \varphi \mid \langle \alpha \rangle \psi.
\]
We have said what the intuitive meaning of the action nominals is.
Formally, the semantics of action nominals are as follows:
\[
 \mathcal{M} \otimes \mathcal{A}, (\omega, \alpha) \vDash j_i \Leftrightarrow \alpha = \alpha_i
\]
Since action nominals can only occur under the scope of at least one action operator, this semantic clause is sufficient.
Notice that action nominals have simple reduction axioms:
\begin{remark}
\label{rem:JRed} The following equivalences hold:
\begin{eqnarray*}
\vDash & \langle \alpha_i \rangle j_i \equiv Pre_{\alpha_i} \\
\vDash & \langle \alpha_i \rangle j_k \equiv \bot & \textup{for } i \neq k.
\end{eqnarray*}
\end{remark}
More importantly, it is now possible to write down a reduction axiom for the quantifier $\exists p$:
\begin{lemma}
\label{lem:MSOProd}
If none of $p_0, \ldots, p_{n-1}$ occur in $\varphi$ or any $Pre_{\alpha_i}$, then:
\[
\vDash \langle \alpha \rangle \exists p. \varphi \equiv \exists p_0 \ldots \exists p_{n - 1}. (\bigwedge_{i \in n} U (p_i \implies Pre_{\alpha_i}) \land \langle \alpha \rangle \varphi(\bigvee_{i \in n}(p_i \land j_i)))
\]
\end{lemma}
\begin{proof}
We can prove this equivalence directly by the following string of equivalences:
\[
 \mathcal{M}, \omega \vDash \langle \alpha \rangle \exists p. \varphi(p)
\]
iff
\[
\mathcal{M} \otimes \mathcal{A}, (\omega, \alpha) \vDash \exists p. \varphi(p)
\]
iff
\[
\exists X \subseteq \{ (\omega', \alpha') \mid \mathcal{M}, \omega' \vDash Pre_{\alpha'}  \}, \mathcal{M} \otimes \mathcal{A}[p \mapsto X], (\omega, \alpha) \vDash \varphi(p)
\]
iff
\[
\exists X \subseteq \bigcup_{i \in n} (\{ \omega' \in \Omega \mid \omega' \vDash Pre_{\alpha_i} \} \times \{ \alpha_i \}): 
\mathcal{M} \otimes \mathcal{A}[p \mapsto X], (\omega, \alpha) \vDash \varphi(p)
\]
iff
\[
\exists X_0 \in \{ \omega' \in \Omega \mid \omega' \vDash Pre_{\alpha_0}\}, \ldots, X_{n - 1} \in \{ \omega' \in \Omega \mid \omega' \vDash Pre_{\alpha_{n - 1}} \}:
\]
\[
\mathcal{M} \otimes \mathcal{A}[p \mapsto \bigcup X_i \times \{\alpha_i\}], (\omega, \alpha) \vDash \varphi(p)
\]
iff
\[
\exists X_0 \in \{ \omega' \in \Omega \mid \omega' \vDash Pre_{\alpha_0}\}, \ldots, X_{n - 1} \in \{ \omega' \in \Omega \mid \omega' \vDash Pre_{\alpha_{n - 1}} \}:
\]
\[
\mathcal{M} \otimes \mathcal{A}[p_0 \mapsto X_i \times \{\alpha_i\}, \ldots, p_{n - 1} \mapsto X_{n - 1} \times \{x_{n - 1}\}], (\omega, \alpha)
\vDash \varphi(\bigvee_{i \in n} p_i)
\]
(where $p_0, \ldots, p_{n - 1}$ are distinct and do not occur in $\varphi$ or any $Pre_{\alpha_i}$), iff
\[
\exists X_0\in\{\omega'\in\Omega \mid \omega' \vDash Pre_{\alpha_0}\}, \ldots, X_{n - 1} \in \{ \omega' \in \Omega \mid \omega' \vDash Pre_{\alpha_{n - 1}} \}:
\]
\[
\mathcal{M} \otimes \mathcal{A}[p_0 \mapsto X_i \times \Sigma, \ldots, p_{n - 1} \mapsto X_{n - 1} \times \Sigma], (\omega, \alpha)
\vDash \varphi(\bigvee_{i \in n} (p_i \land j_i))
\]
iff
\[
\exists X_0\in\{\omega'\in\Omega \mid \omega' \vDash Pre_{\alpha_0}\}, \ldots, X_{n - 1} \in \{ \omega' \in \Omega \mid \omega' \vDash Pre_{\alpha_{n - 1}} \}:
\]
\[
\mathcal{M}[p_0 \mapsto X_i, \ldots, p_{n - 1} \mapsto X_{n - 1}], \omega \vDash \langle \alpha \rangle \varphi(\bigvee_{i \in n}(p_i \land j_i))
\]
iff
\[
\exists X_0, \ldots, X_{n - 1} \subseteq \Omega:
\mathcal{M}, \omega \vDash \bigwedge_{i \in n} U(p_i \implies Pre_{\alpha_i}) \land \langle \alpha \rangle \varphi(\bigvee_{i \in n}(p_i \land j_i))
\]
iff
\[
\mathcal{M}, \omega, \vDash \exists p_0, \ldots, \exists p_{n - 1}.
(\bigwedge_{i \in n} U(p_i \implies Pre_{\alpha_i}) \land \langle \alpha \rangle \varphi(\bigvee_{i \in n}(p_i \land j_i)))
\]
\end{proof}

These reduction axioms
(from Remark \ref{rem:JRed} and Lemma \ref{lem:MSOProd})
are what we use to establish the product closure of $\mathcal{L}_{\exists, U}$.

\begin{theorem}
\label{thm:prod}
$\mathcal{L}_{\exists, U}$ is computably closed for product update.  I.e.:
\[
\forall \alpha \in \Sigma, \forall \psi \in \mathcal{L}_{\exists, U}, \exists \chi \in \mathcal{L}_{\exists, U}:~\vDash \chi \equiv \langle \alpha \rangle \psi\textup,
\]
where $\chi$ is effectively computable from $\psi$ and $(\Sigma, \alpha)$.
\end{theorem}

Theorem \ref{thm:prod} follows from Lemma \ref{lem:jElim} below, which is phrased in terms of the set of sentences $\mathcal{S}_{\exists, U, j}$,
defined recursively as:
\[
 \varphi ::= p \mid j_i \mid \neg \varphi \mid \varphi \land \varphi \mid \square \varphi \mid \exists p. \varphi \mid U \varphi
\]
It is in terms of $\mathcal{S}_{\exists, U, j}$, which clearly contains $\mathcal{L}_{\exists, U}$,
that we are able to prove a reduction lemma, Lemma \ref{lem:jElim} below.
Since $\mathcal{S}_{\exists, U, j}$ clearly contains $\mathcal{L}_{\exists, U}$, then Theorem \ref{thm:prod} is an immediate corollary of Lemma \ref{lem:jElim}.

\begin{lemma}
\label{lem:jElim}
\[
\forall \alpha \in \Sigma, \forall \psi \in \mathcal{S}_{\exists, U, j}, \exists \chi \in \mathcal{L}_{\exists, U}:~\vDash \chi \equiv \langle \alpha \rangle \psi
\]
\end{lemma}
\begin{proof}
We would like to prove this by induction on $\psi$, by using our reduction axioms, and using the following inductive hypothesis, where $\gamma$ is less complex than $\psi$:
\begin{equation}
\label{eqn:IH2}
 \exists \delta\in \mathcal{S}_{\exists, U, j}:~\vDash \delta \equiv \langle \alpha \rangle \gamma
\end{equation}
First we need an appropriate notion of complexity:
since at the new case to be treated at the inductive stage, the propositional quantifier case, the reduction axion \emph{increases} the complexity of the formula, by the standard definition of complexity,
we suppose we have a definition according to which formulas with greater quantifier depth have greater complexity,
and that within that stratification the standard notion of formula complexity applies.
Let $\varepsilon(\psi)$ be the number of propositional existential quantifiers $\exists$ in $\psi$.

Notice that we can immediately prove (\ref{eqn:IH2}) for the case when $\varepsilon(\psi) = 0$, for then $\psi \in \mathcal{L}_{U, j}$, so the new cases to be treated are just (a) $\psi:= j_i$ with $\alpha = \alpha_i$ and (b) $\psi:= j_i$ with $\alpha \neq \alpha_i$.  Here we use the `reduction axioms' from Remark \ref{rmk:JRed}: for (a), $\vDash \langle \alpha \rangle j_i \equiv \top$, and for (b) $\vDash \langle \alpha \rangle j_i \equiv \bot$.

Therefore it will suffice, to prove the Lemma, to suppose that:
\begin{equation}
\label{eqn:IH3}
\textup{(\ref{eqn:IH2}) holds for }\varepsilon(\psi) < k
\end{equation}
holds, and then to show that (\ref{eqn:IH2}) holds for $\varepsilon(\psi) = k$ ($k \in \mathbb{N}^+$).

Again, the `old' cases are known, so we go straight to the new case:

Let $\psi:= \exists p. \gamma(p)$.
Here we choose some $q_0, \ldots, q_{n-1}$ not occurring in either $\gamma$ or any of the $PRE_\alpha$'s,
and letting $\gamma' = \gamma(\bigvee_{i \in n}(q_i \land j_i))$, define $\varphi$ as follows:
\[
\varphi:= \exists q_0 \ldots \exists q_{n - 1}. (\bigwedge_{i \in n} U(q_i \implies Pre_{\alpha_i}) \land \langle \alpha \rangle \gamma'). 
\]
Then by Lemma \ref{lem:MSOProd}, $\vDash \varphi \equiv \langle \alpha \rangle \psi$.
Furthermore, notice that $\gamma' \in \mathcal{S}_{\exists, U, j}$,
and that $\varepsilon(\gamma') = \varepsilon(\psi) - 1 < \varepsilon(\psi)$.
So we can invoke our inductive hypothesis (\ref{eqn:IH3}),
establishing that there exists $\delta \in \mathcal{L}_{\exists, U}$ such that:
\begin{equation}
\label{eqn:lastEquiv}
\vDash \delta \equiv \langle \alpha \rangle \gamma'. 
\end{equation}
It immediately follows that $\vDash \theta \equiv \varphi$, where:
\[
\theta :=  \exists q_0 \ldots \exists q_{n - 1}. (\bigwedge_{i \in n} U (q_i \implies Pre_{\alpha_i}) \land \delta)
\]
Then since $\theta \in \mathcal{L}_{\exists, U}$ we are done.
\end{proof}

\section{A corollary for $\mathcal{L}_\exists$}
\label{sec:qml}

The language $\mathcal{L}_\exists$ is the language $\mathcal{L}_{\exists, U}$ without the global modality $U$.
The absence of a global modality means that $\mathcal{L}_\exists$ is a `local' language:
as observed in \cite{vB_MLCL},
$\mathcal{L}_\exists$ is invariant under generated submodels.
\cite{tC_SOPML} shows further that it is a \emph{proper} fragment of the generated-submodel-invariant sublanguage of $\mathcal{L}_{\exists, U}$
and indeed establishes a second-order analogue of the van Benthem-Rosen theorem,
characterising the expressivity of $\mathcal{L}_{\exists}$ as a fragment of $MSO$.
We need some more definitions in order to formulate the characterisation given in \cite{tC_SOPML} of $\mathcal{L}_\exists$ as a fragment of $\mathcal{L}_{\exists, U}$:
Firstly, for $k \in \mathbb{N}$, and any point $\omega$ in the domain of the model $\mathcal{M}$,
we denote by $\mathcal{M}^k_\omega$ the submodel of $\mathcal{M}$ generated from going at most $k$ steps along $R$ from $\omega$.
Let $\mathbb{K}(\varphi)$ denote the class of pointed models defined by $\varphi$,
i.e.~such that for any pointed\footnote{A `pointed model' is a model $(\Omega, R, V)$ together with a point $\omega \in \Omega$.} model $(\mathcal{M}, \omega)$, we have $(\mathcal{M}, \omega) \in \mathbb{K} \Leftrightarrow \mathcal{M}, \omega \vDash \varphi$.
We say that $\varphi$ has \df{degree} $k$ (for $k \in \mathbb{N}$) when
for all pointed models $(\mathcal{M}, \omega)$,
$(\mathcal{M}, \omega) \in \mathbb{K}(\varphi) \Leftrightarrow (\mathcal{M}^k_\omega, \omega) \in \mathbb{K}(\varphi)$.
If there is some $k \in \mathbb{N}$ such that $\varphi$ has degree $k$, then let $deg(\varphi)$ denote the least such $k$, and $\infty$ otherwise.
Then the following theorem characterises $\mathcal{L}_\exists$ as a fragment of $\mathcal{L}_{\exists, U}$:
\begin{theorem}[{\cite[Theorem 6]{tC_SOPML}}]
\label{thm:tC}
A formula $\varphi \in \mathcal{L}_{\exists, U}$ is equivalent to a formula of $\mathcal{L}_\exists$ just if $deg(\varphi) \neq \infty$
\end{theorem}

We now want to show that having finite degree is a property that is preserved by adding action modalities:

\begin{lemma}
\label{lem:findeg}
If $\varphi$ and all of $PRE_{\alpha_0}, \ldots, PRE_{\alpha_n}$ have finite degree then  $\langle \alpha \rangle \varphi$ has finite degree.
\end{lemma}
\begin{proof}
Notice that
\begin{equation}
\label{eqn:presdeg}
\textup{for any }m \leq k\textup{ then if $\varphi$ has degree $m$, $\varphi$ also has degree $k$.}
\end{equation}
Let $k^\star = max(\{deg(PRE_{\alpha_0}), \ldots, deg(PRE_{\alpha_n}) \} + deg(\varphi)$.
Then (\ref{eqn:presdeg}) is used to show that
\[(\mathcal{M} \otimes \mathcal{A})_{(\omega, \alpha)}^{deg(\varphi)} \textup{ is a submodel of } (\mathcal{M}_\omega^{k^\star} \otimes \mathcal{A}),\]
which is in turn used to prove that
$deg(\langle \alpha \rangle \varphi) \leq k^\star$.
\end{proof}

Then as a corollary of Theorem \ref{thm:prod} we also obtain the closure of the propositionally quantified modal language $\mathcal{L}_\exists$:

\begin{corollary}
\label{cor:qml}
$\mathcal{L}_\exists$ is closed for product update.  I.e.:
\[
\forall \alpha \in \Sigma, \forall \psi \in \mathcal{L}_{\exists}, \exists \chi \in \mathcal{L}_{\exists}:~\vDash \chi \equiv \langle \alpha \rangle \psi\textup,
\]
\end{corollary}
\begin{proof}
Take $\varphi \in \mathcal{L}_{\exists}$ and consider $\langle \alpha \rangle \varphi$; since $\mathcal{L}_{\exists} \subset \mathcal{L}_{\exists, U}$, then by Theorem \ref{thm:prod} there is a formula $\psi' \in \mathcal{L}_{\exists, U}$ that is equivalent to $\langle \alpha \rangle \varphi$.
Since $\varphi \in \mathcal{L}_\exists$, we know that $\varphi$ has finite degree,
and then, by Lemma \ref{lem:findeg} that $\langle \alpha \rangle \varphi$ also has finite degree.
Therefore, since $\psi'$ and $\langle \alpha \rangle \varphi$ are equivalent, $\psi'$ has finite degree, and therefore by Theorem \ref{thm:tC} is equivalent to some $\psi \in \mathcal{L}_\exists$.
\end{proof}

\section{A corollary for $\mathcal{L}_\mu$}

In this section we observe that, similarly, the relativisation closure of the modal fixpoint language $\mathcal{L}_\mu$ is a corollary of our Theorem \ref{thm:prod}.  The closure of $\mathcal{L}_\mu$ is already obtained in \cite{vBI_PDLMuClosed}.  The proof of those authors is more direct, but our proof, just as for $\mathcal{L}_\exists$ in the previous section, is very quick because it again uses a semantic characterisation result.
$\mathcal{L}_\mu$ is defined as follows:
\[
 \varphi ::= p \mid \neg \varphi \mid \varphi \land \varphi \mid \square \varphi \mid \nu p. \varphi,
\]
where crucially in $\nu p. \varphi$, $\varphi$ must be \emph{positive} in $p$ (i.e. all occurrences of $p$ must be under the scope of an even number of negations $\neg$).  The new semantic clause for $\nu$ is as follows:
\[
\llbracket \nu p. \varphi \rrbracket =
\bigcup \{ X \subseteq \Omega \mid \forall \omega \in X, \mathcal{M}[p \mapsto X], \omega \vDash \varphi \}
\]
The $\nu$ operator expresses \emph{greatest fixpoints}.%
\footnote{More specifically: $\llbracket \nu p. \varphi(p) \rrbracket$ is the greatest fixpoint of the function $F_\varphi: 2^\Omega \rightarrow 2^\Omega$ defined by $F_{\varphi}(X) = \llbracket \varphi(p) \rrbracket_{\mathcal{M}[p \mapsto X]}$.  The requirement that $\varphi$ be positive in $p$ ensures that $F_{\varphi}$ is monotone and therefore has, as a consequence of the Knaster-Tarski theorem, a greatest fixpoint.}
It is easy to see that $\mathcal{L}_\mu$ is a strict fragment of $\mathcal{L}_{\exists, U}$, since we can define $\nu p. \varphi$ as
\[
\exists p (p \land U(p \implies \varphi)).
\]
Theorem \ref{thm:JW} below is a more remarkable result, precisely characterising $\mathcal{L}_\mu$ in terms of $\mathcal{L}_{\exists, U}$ is proved in \cite{JW_MuIsBis}.
We say that a modal formula $\varphi$ is \df{bisimulation-invariant} just if
for any pointed models $(\mathcal{M}, \omega)$ and $(\mathcal{M}', \omega')$ that are bisimilar%
\footnote{For a definition of the fundamental modal logic notion of bisimulation, see e.g.\ \cite{BdRV_ML}.}%
, $\mathcal{M}, \omega \vDash \varphi \Leftrightarrow \mathcal{M}', \omega' \vDash \varphi$.
\begin{theorem}
\label{thm:JW}
A formula $\varphi \in \mathcal{L}_{\exists, U}$ is equivalent to an $\mathcal{L}_\mu$-formula iff it is bisimulation-invariant.
\end{theorem}
\begin{proof}Theorem \ref{thm:JW} follows immediately from \cite[Theorem 11]{JW_MuIsBis}, given that $\mathcal{L}_\mu$ formulas are bisimulation-invariant.\end{proof}
That elegant result is what leads us to see that $\mathcal{L}_\mu$ is also closed for product update.  We just need one additional lemma, which says that adding product-update modalities does not break bisimulation-invariance:
\begin{lemma}
\label{Lem:ProdBis}
If $\varphi$ and all of $PRE_{\alpha_1}, \ldots, PRE_{\alpha_n}$ are bisimulation-invariant then  $\langle \alpha \rangle \varphi$ is bisimulation-invariant.
\end{lemma}
\begin{proof}
Take any bisimilar models $\mathcal{M}$ and $\mathcal{M}'$, and a point $\omega$ and $\omega'$ in the domain of each.
It will suffice to show that $\mathcal{M}, \omega \vDash \varphi$ iff $\mathcal{M}', \omega' \vDash \varphi$.
Let $Z$ be a bisimulation for those models, with $\omega Z \omega'$.
Using the fact that the $PRE_{\alpha_i}$'s are bisimulation-invariant,
the relation $Y$ defined as follows can be shown to be a bisimulation between
$\mathcal{M}\otimes \mathcal{A}$ and $\mathcal{M}' \otimes \mathcal{A}$:
\[
(s, \beta) Y (t, \gamma) \, \dfn{\Leftrightarrow} \ s Z t \ \& \ \beta = \gamma.
\]
Then since $\varphi$ is also bisimulation-invariant, we know that
$\mathcal{M} \otimes \mathcal{A}, (\omega, \alpha) \vDash \varphi$ iff $\mathcal{M}' \otimes \mathcal{A}, (\omega', \alpha) \vDash \varphi$,
which is equivalent to what we set out to establish.
\end{proof}

\begin{corollary}
\label{cor:mu}
$\mathcal{L}_\mu$ is closed for product update, i.e.:
\[
\forall \alpha \in \Sigma, \forall \psi \in \mathcal{L}_{\mu}, \exists \chi \in \mathcal{L}_{\mu}:~\vDash \chi \equiv \langle \alpha \rangle \psi\textup,
\]
\end{corollary}
\begin{proof}
The proof is analogous to that of Corollary \ref{cor:qml}.
\end{proof}

\section{Conclusion}

We have shown that monadic second-order logic is closed for product update.
We did this using `action nominals', which served as a sort of `memory' in the reduction axiom process.
This book-keeping allowed us to keep track of which action had occurred,
and therefore to `talk about' subsets of the product model
in terms of sequences of subsets of the initial model.
Although we pass via a language with action nominals,
Lemma \ref{lem:jElim} does specify a method that, given a formula $\psi \in \mathcal{L}_{\exists, U}$
yields a concrete formula $\chi \in \mathcal{L}_{\exists, U}$ that is equivalent to $\langle \alpha \rangle \psi$.

We obtained as corollaries the facts that propositionally quantified modal logic,
and the modal fixpoint calculus, are closed for product update.
This was in each case achieved by using a semantic characterisation result.
The non-constructive nature of our proofs of these corollaries, unlike that of Lemma \ref{lem:jElim},
mean that they do not establish that either propositionally quantified modal logic or the modal fixpoint calculus are \emph{computably} closed.

\section*{Acknowledgements}
Participants in the ILLC's Dynamic Logic seminar series provided helpful comments,
especially Johan van Benthem, who suggested Corollary \ref{cor:mu}.

\bibliographystyle{named}
\bibliography{biblio.bib}

\end{document}